\documentclass[11pt]{amsart}

\usepackage{amsmath,amstext,amsthm,amsfonts,amssymb,psfrag}

\theoremstyle{plain}

\newtheorem{theorem}{Theorem}[section]
\newtheorem{lemma}[theorem]{Lemma}

\newtheorem{proposition}[theorem]{Proposition}
\theoremstyle{remark}
\newtheorem{definition}{Definition}
\newtheorem{remark}[theorem]{Remark}

\newtheorem{example}[theorem]{Example}

\title{ON ZARISKI'S MULTIPLICITY CONJECTURE}
\author{MAHDI TEYMURI GARAKANI}

\address{IMPA - Instituto de Matem\'atica Pura e Aplicada,
Estrada Dona Castorina 110, Jardim Botanico 22460-320 Rio de Janeiro- RJ Brazil}

\email{teymuri@impa.br}



\begin{document}

\maketitle

\section{Introduction}
Since the early 1960's, O. Zariski developed a comprehensive theory of
equisingularity in codimension one. He initiated an equisingularity program
with topological, differential geometrical and purely algebraical point
of view
and proposed a problem list in [22] as an extraction of many possible
conjectures
in singularity theory [23]. In this part we will be concerned with
topological aspects of this program and more specifically with the so-called
Zariski's multiplicity conjecture.
We first recall some definitions.
Let $ f , g: (\mathbb{C}^{n},0) \rightarrow (\mathbb{C},0)$ be two germs of
holomorphic functions and $V_{f}$ and $V_{g}$ be two germs at the
origin of the hypersurfaces
defined by $f^{-1}(0)$ and $g^{-1}(0)$ respectively.
We suppose $0 \in \mathbb{C}^{n}$ is an isolated singularity of the functions.
The \textit{algebraic multiplicity} $m_f$ of the germs of $V_{f}$ or
$f$ is the
order of vanishing of function $f$ at $0 \in \mathbb{C}^{n}$ or
equivalently is
the order of the first nonzero leading term in the Taylor expansion of $f$
\begin{eqnarray}
f=f_\nu + f_{\nu + 1} + \cdots \nonumber
\end{eqnarray}
where $f_i$ is homogeneous polynomial of degree $i$.

\begin{definition} We say $V_{f}$ and $V_{g}$
are
\textit{topologically equisingular} or topologically V-equivalent if there
is a germ of homeomorphism $\phi : (\mathbb{C}^{n},0) \rightarrow
(\mathbb{C}^{n},0)$ sending $V_{f}$ onto $V_{g}$. More precisely, there
are neighborhoods $U$ and $U'$ of $0 \in \mathbb{C}^{n}$  such that $f$
and $g$
are defined and a homeomorphism
$\phi : U \rightarrow U'$ such that $\phi (f^{-1}(0) \cap U)= g^{-1}(0)
\cap U'$
and $\phi(0)=0$.
\end{definition}

\noindent
\textbf{Zariski conjecture.} \textit{Topological equisingularity of germs of
hypersurfaces implies equimultiplicity.}\\

A well known result by Burau [4] and Zariski [23] states an affirmative
answer in
the case of curves ($n=2$). In higher dimension the conjecture is still open
despite more than three decades effort to prove it.

Here we discuss some features of the problem, especially the relations of
the work of A'Campo on the zeta function of a monodromy and the
Zariski's multiplicity conjecture. Also some previous results are
sharpened; the results of [6] and [18] in theorem (3.2) and the one in
[7] in theorem (5.2). In an analogy with hypersurfaces, J.F. Mattei
asked the same question about multiplicity for holomorphic foliations.
In section (6) we recall some remarkable results for foliations.

\textbf{Acknowledgment.} This work was partially done during a visit of
the author to the Max Planck Institute. The author would like to
express his deepest gratitude to Professor Matilde Marcolli who made
this visit possible. Also I would like to thank the referee for useful
comments.

\section{Preliminaries}
In [15], Milnor has opened a beautiful account on the complex
hypersurfaces. The main achievement of it, is the Milnor fibration
which we mention here. Also we recall briefly some generalities about
complex hypersurfaces.

Let $f:U \subset \mathbb{C}^{n} \rightarrow \mathbb{C}$ be a holomorphic
function on an open neighborhood of $0$ in $ \mathbb{C}^{n}$ and $f(0)=0$. We
denote
$D_{\epsilon}=\lbrace z \vert z \in \mathbb{C}^{n} : \Vert z \Vert \leq
\epsilon
\rbrace$, $ S_{\epsilon}= \partial D_{\epsilon}$, $H_0=\lbrace z \in
\mathbb{C}^{n} \vert f(z)=0 \rbrace$
and $d_zf=(\frac{{\partial}f}{{\partial}z_1}(z), \cdots ,
\frac{{\partial}f}{{\partial}z_n}(z))$.

We say the origin is an isolated singularity of $f$ if $d_0f=0$ and $d_zf \neq
0$ for
a neighborhood of $0 \in \mathbb{C}^{n}$ except $0$.

Let $\mathcal{O}_n$ be the ring of germs of holomorphic functions defined in
some neighborhood of $0 \in \mathbb{C}^{n}$ and let
$<\frac{{\partial}f}{{\partial}z_1}, \cdots ,
\frac{{\partial}f}{{\partial}z_n}>$ be the ideal generated by the germs at $0
\in \mathbb{C}^{n}$ of derivative components of $f$. We define \textit{Milnor
number} $\mu $ of the holomorphic function $f$ at
$0 \in \mathbb{C}^{n}$ as
\begin{eqnarray}
\mu = dim_{\mathbb{C}}{\mathcal{O}_n}/{<\frac{{\partial}f}{{\partial}z_1},
\cdots , \frac{{\partial}f}{{\partial}z_n}>} \nonumber
\end{eqnarray}
This number is finite and nonzero if and only if $0 \in \mathbb{C}^{n}$ is an
isolated singularity of $f$, a hypothesis which we will assume from now on. In
this case $\mu $ coincides with the topological degree of the Gauss mapping
induced by $d_zf $ on $ S_{\epsilon} $ for $\epsilon $ small enough. The
following lemma is useful to deal with the Milnor number.
\begin{lemma}
Let $0 < \mu < \infty$. Given $ \epsilon > 0 $ there exists $ \delta > 0 $ such
that for any $c \in \mathbb{C}^{n}$ with $\Vert c \Vert < \delta $ the
number of solutions of the equation $d_zf=c$ in the ball $ D_{\epsilon}
$ is at
most $\mu $. Moreover, if $ p_1, \cdots , p_m $, $m \leqslant \mu $, are such
solutions, then
$\sum_{i=1}^{m} \mu (f-\sum_{i=1}^{n}z_ic_i, p_i)= \mu.$
\end{lemma}
The following theorem is called the Milnor fibration theorem:

\begin{theorem}
For $ \epsilon $ small enough the mapping $\psi_{\epsilon}: S_{\epsilon}
\setminus H_0 \rightarrow S^1$ defined by
$\psi_{\epsilon}(z)={f(z)}/{\Vert f(z) \Vert}$ is a smooth fibration which is
called the Milnor fibration. Moreover the fibers of $\psi_{\epsilon} $
have the
homotopy type of a bouquet of $\mu $ (the Milnor number of the holomorphic
function $f$ at
$0 \in \mathbb{C}^{n}$) spheres of dimension $n-1$.
\end{theorem}

Also we call the number of spheres, the number of \textit{vanishing cycles} of
$f$ at $0$. The following theorem is due to L\^e [14]:

\begin{theorem}
If $V_{f}$ and $V_{g}$ are topologically equisingular then the number of
vanishing cycles at $0$ of $f$ and $g$ are the same.
\end{theorem}

Now we recall some definitions and facts about deformations of functions. A
\textit{deformation} of a holomorphic function $ f: (\mathbb{C}^{n},0)
\rightarrow (\mathbb{C},0)$ is a family $(f_t)_{t \in [0,1]}$ of germs of
holomorphic functions with isolated singularities at $0 \in
\mathbb{C}^{n}$. The
\textit{jump of the family} $(f_t)$ is
$\mu(f_0)-\mu(f_t)$, where $\mu$ is the Milnor number at the origin. It is
independent of $ t $ for $ t $ small enough, moreover by the upper
semi-continuity
of $\mu$ this number is a non-negative integer.

We use frequently the following theorem proved by L\^e and Ramanujam [12]:
\begin{theorem}
Let $(f_s)_{s \in [0,1]}$ be a $C^{\infty}$ family of hypersurfaces having an
isolated singularity at the origin. If the Milnor number of singularity does
not change then the topological type of singularity does not change too
provided that
$n\neq3$.
\end{theorem}

In [11], theorem (2.4) is generalized which we recall in section (5).
Finally we recall an interesting result of P. Samuel [19]:

\begin{theorem}
Every germ $V_f$ is analytically equivalent with $V_g$ in which $g$ is a
polynomial.
\end{theorem}

Moreover we may choose a polynomial $g$ with cutting the Taylor expansion
of $f$ in somewhere.
By the theorem of (2.5) it is enough to consider polynomials to prove the
conjecture.
\section[Topological right equivalency]{The topological right
equivalent complex hypersurfaces} In this section we recall several
ways to define a topological type of a holomorphic function and
relations between them according to [10], [17] and [20].

Let $ f, g : (\mathbb{C}^n,0) \rightarrow (\mathbb{C},0) $ be two germs
of holomorphic functions with an isolated singularity at the origin.
\begin{definition}
$f$ and $g$ are topologically right equivalent if there is a germ of
homeomorphism
$ \varphi: (\mathbb{C}^n,0) \rightarrow (\mathbb{C}^n,0) $ satisfying
$f=g \circ \varphi $
\end{definition}
\begin{definition}
$f$ and $g$ are topologically right-left equivalent if there are germs
of homeomorphism $\varphi: (\mathbb{C}^n,0) \rightarrow
(\mathbb{C}^n,0)$ and
$\psi: (\mathbb{C},0) \rightarrow (\mathbb{C},0) $ satisfying
$f=\psi \circ g \circ \varphi $
\end{definition}
Put $V_f:=f^{-1}(0)$. By [15], $S_{\varepsilon}^{2n-1} \cap V_f$ is a smooth
$ (2n-3) $-dimensional manifold for $\varepsilon > 0$ sufficiently
small. The pair
$(S_{\varepsilon}^{2n-1} , S_{\varepsilon}^{2n-1} \cap V_f )$ is called
the link of the singularity of $f$.
\begin{definition}
$f$ and $g$ are link equivalent if
$(S_{\varepsilon}^{2n-1} , S_{\varepsilon}^{2n-1} \cap V_f )$ is
homeomorphic to
$(S_{\varepsilon}^{2n-1} , S_{\varepsilon}^{2n-1} \cap V_g)$ for all
sufficiently small  $\varepsilon$.
\end{definition}
By the definitions, the right equivalence implies the right-left
equivalence, which in turn implies the V-equivalence. The outstanding
result, obtained by King [10] in $n \neq 3$ and by Perron [17] in $ n=3
$, is the following:
\begin{theorem}
The topological V-equivalence implies topologically right-left
equivalence. Moreover
if $f$ and $g$ are topologically right-left equivalent then $g$ is
topologically right equivalent either to $f$ or to $\overline{f}$, the complex
conjugate of $f$.
\end{theorem}
Using theorem (3.1), Risler and Trotman in [18] proved that right-left
bilipschitz equivalence implies equimultiplicity.

Since $(D_{\varepsilon}^{2n-1} , D_{\varepsilon}^{2n-1} \cap V_f )$ is
homeomorphic to the cone cover over the link
$(S_{\varepsilon}^{2n-1} , S_{\varepsilon}^{2n-1} \cap V_f )$ ([15]),
the link equivalence implies the V-equivalence. Conversely Saeki [20]
showed that the topological V-equivalence implies the link
equivalence. This means that there is a homeomorphism $\varphi_1:
(\mathbb{C}^n,0) \rightarrow (\mathbb{C}^n,0)$, sending
$V_f$ onto $V_g$ and such that $\vert \varphi_1(z) \vert=\vert z \vert$.
By theorem (3.1) there is a homeomorphism $\varphi_2: (\mathbb{C}^n,0)
\rightarrow (\mathbb{C}^n,0)$ such that $\vert f(z)\vert = \vert g
\circ \varphi_2 (z)\vert$.
Comte, Milman and Trotman [6] showed that if there is a germ of homeomorphism
$\varphi: (\mathbb{C}^n,0) \rightarrow (\mathbb{C}^n,0)$ having
\textit{simultaneously} the properties of $\varphi_1$ and of
$\varphi_2$ then the multiplicity conjecture is true. In fact they proved that
it suffices to assume that there are positive constants $A$, $B$, $C$ and
$D$ such that:

(1) $A \vert z \vert \leq \vert \varphi (z) \vert \leq B \vert z
\vert$, for all $z$ near 0, and

(2) $C\vert f(z) \vert \leq \vert g \circ \varphi (z)\vert \leq D \vert
f(z)\vert$, for all $z$ near 0.

Now we prove that it is enough to assume the conditions (1) and (2) are
valid for some
special sequences converge to the origin. Given two holomorphic function germs
$f,g: (\mathbb{C}^n,0) \rightarrow (\mathbb{C},0)$, by an analytic
change of coordinates
one may assume that the $z_1$-axis is not contained in the tangent
cones $C(V_f), C(V_g)$ (respectively the zero set of first non zero jet
of $f$ and $g$), so that $f(z_1,0, \cdots , 0) \neq 0$ and $g(z_1,0,
\cdots , 0) \neq 0$ for a neighborhood of 0 in the $z_1$-axis, and by
theorem (2.5) one may assume $f$ and $g$ are polynomials. In this
situation we have the following:
\\
\begin{theorem}
Suppose there are a germ of homeomorphism $\varphi: (\mathbb{C}^n,0)
\rightarrow (\mathbb{C}^n,0) $ with inverse $\psi$ and positive
constants $A$, $B$, $C$ and $D$ and two sequences $w_m$ and ${{w'}}_m$
in the $z_1$-axis which converge to the origin with the following properties:

(i) $ \vert \psi ({w'}_m ) \vert \leq A \vert w'_m \vert $,
$ \vert \varphi ({w}_m ) \vert\leq B \vert w_m \vert$ and

(ii) $C\vert f({w}_m) \vert \leq \vert g \circ \varphi ({w}_m)\vert $,
$D \vert g({w'}_m) \vert \leq \vert f \circ \psi ({w'}_m) \vert $

then $ m_f=m_g $.
\end{theorem}

The conditions (i) and (ii) are slightly weaker than conditions (1) and
(2) above.
\begin{proof}
Write $$ f(z)=f_k(z)+f_{k+1}(z)+ \cdots +f_{k+r}(z),  $$
$$ g(z)=g_l(z)+g_{l+1}(z)+ \cdots +g_{l+s}(z). $$ $f_i$ and $g_j$ are
homogeneous parts of degree $i$ and $j$ of $f$ and $g$ respectively. $f_k$ and
$g_l$ are not identically zero. We want to prove $k=l$. By contrary suppose
$l>k$. The other case is similar. Let $w_1=(z_1,0, \cdots ,0)$ and
$w_m=(t_mz_1,0, \cdots ,0)$,
$t_m \neq 0$ and converges to the origin. Also write $g$ in the
following form:
$$g(z)=\sum^{l+s}_{j=l}\sum_{\vert \beta \vert=j}C^j_{\beta} z^{\beta},$$
where $z=(z_1,\cdots,z_n)$ and $\beta=(\beta_1,\cdots,\beta_n)$,
$\beta_i \in \mathbb{N} \cup \{0\}$. Now we have
$$f(w_m)=f_k(w_m)+f_{k+1}(w_m)+ \cdots +f_{k+r}(w_m)~~ or$$
$$f(w_m)=t^k_m[f_k(w_1)+t_mf_{k+1}(w_1)+ \cdots +t^{r}_mf_{k+r}(w_1)] $$ and
$$\vert (g \circ \varphi)(w_m) \vert \leq  \sum^{l+s}_{j=l}\sum_{\vert
\beta \vert=j}
\vert C^j_{\beta} \vert B^j {\vert t_mz_1 \vert}^j$$ by (i). Now we use
condition (ii). It is:
$$  C\vert f(w_m) \vert\leq \vert g \circ \varphi ({w}_m)\vert ~~ or$$
$$ C\vert t^k_m [ f_k(w_1)+t_mf_{k+1}(w_1)+ \cdots +t^{r}_mf_{k+r}(w_1)
] \vert \leq \sum^{l+s}_{j=l}\sum_{\vert \beta \vert=j}
\vert C^j_{\beta} \vert B^j {\vert t_mz_1 \vert}^j.$$
Divided two sides of above inequality by $ \vert t^k_m \vert$ we obtain
the following:
$$ C\vert f_k(w_1)+t_mf_{k+1}(w_1)+ \cdots +t^{r}_mf_{k+r}(w_1) \vert
\leq \sum^{l+s}_{j=l}\sum_{\vert \beta \vert=j}
\vert C^j_{\beta} \vert B^j {\vert t_mz_1 \vert}^{j-k},$$
or
$$ C\vert f_k(w_1) \vert \leq \sum^{l+s}_{j=l}\sum_{\vert \beta \vert=j}
\vert C^j_{\beta} \vert B^j {\vert t_mz_1 \vert}^{j-k} +
C \vert t^{1}_mf_{k+1}(w_1)+ \cdots +t^{r}_mf_{k+r}(w_1) \vert.$$
When $t_m$ goes to zero, the right hand of the last inequality goes to
zero but the left hand is a positive constant. This contradiction shows
$l=k$.
\end{proof}
\section{The zeta function of a monodromy}
Now we recall some features from [1] and [2]. Let $ f: \mathbb{C}^n
\rightarrow \mathbb{C}$ be a polynomial so that $f(0)=0$ and consider
the hypersurface defined by it, $V_f=f^{-1}(0)$. The map
$$\pi : z \in  S_{\varepsilon}^{2n-1} \backslash V_f \longmapsto
arg(f(z)) \in S^1,$$
defines a Milnor fibration of the hypersurface $V_f$ at the origin. The
fibre $F_{\theta}= \pi^{-1}(\theta)$, $ \theta \in S^1$, is a
$2(n-1)$-dimensional differential manifold and the characteristic
homeomorphism of this fibration
$$h: F_{\theta} \rightarrow F_{\theta}$$
is the geometric monodromy of $V_f$ at the origin. By definition the
zeta function of $h$ is the following:
$$Z(t)=\prod_{q \geq 0}
\{\mathrm{det}(Id^*-th^*;H^q(F_{\theta},\mathbb{C}))\}^{(-1)^{q+1}}.$$
When the origin of $\mathbb{C}^n$ is an isolated singular point of
$V_f$, one has
\begin{displaymath}
H^q(F_{\theta},\mathbb{C})=\left\{ \begin{array}{lll}
            \mathbb{C} & \textrm{ $q=0$}\\
          0 & \textrm{$q \neq 0, q \neq n$}\\
          \mathbb{C}^{\mu} & \textrm{$q =n$,}
          \end{array} \right.
\end{displaymath}

where $\mu$ is the Milnor number of $f$ and therefore the
characteristic polynomial $\Delta(t)$ of the monodromy at degree $n$ is
deduced from the zeta function $Z(t)$ by the formula
$$\Delta(t)=t^{\mu}[\frac{t-1}{t}Z(\frac{1}{t})]^{(-1)^{n+1}}. $$
For an integer $k \geq 1$; let the integer number
$$\Lambda(h^k)= \sum_{q \geq 0}(-1)^q
\mathrm{Trace}[(h^*)^k;H^q(F_{\theta},\mathbb{C})] $$
the Lefschetz number of the $k$-th power of $h$. Let $s_1,s_2, \cdots$ be
the integers defined by the following recurrence relations:
$$\Lambda(h^k)=\sum_{i \vert k}s_i,$$ $k \geq 1$, then the zeta function of
$h$ is given by
$$ Z(t)=\prod_{q \geq 0}(1-t^i)^{\frac{-s_i}{i}}.$$

The Lefschetz numbers $\Lambda(h^k)$ are topological invariants of the
singularity of $V_f$, therefore the integers $s_1,s_2, \cdots$ are
topological invariants.

\begin{remark}
In [2], A'Campo has calculated $\Lambda(h)$ as following:

\begin{displaymath}
\Lambda(h)= \left\{ \begin{array}{ll}
           0 & \textrm{if $d_0f=0$}\\
          1 & \textrm{if $d_0f \neq 0$.}
          \end{array} \right.
\end{displaymath}
\end{remark}

This tells us that if $f$ is regular and $g$ is singular at the origin
there is no topological equivalence between germs of $V_f$ and $V_g$ at
the origin.

\begin{remark}
More generally Deligne has explained in a letter to A'Campo (see [1],
[9]) that
$$ \Lambda(h^k)=0,~~~ \mathrm{if}~~~ 0<k<~~~ \mathrm{multiplicity~~ of
~~}V_f~~~
\mathrm{at~~~the~~~ origin}. $$
\end{remark}

The Lefschetz numbers $\Lambda(h^k)$ are topological invariants of the
singularity of $V_f$, therefore the integers $s_1,s_2, \cdots$ are
topological invariants. A'Campo discovered the meaning of the
topological invariants
$s_1,s_2, \cdots$ as following:

Let $ \pi: X \rightarrow \mathbb{C}^n $ be a proper modification such
that in all
points of $ S:={\pi}^{-1}(0) $, the divisor $V'_f:={\pi}^{-1}(V_f)$ has
normal crossings. Such a local resolution of $(\mathbb{C}^n,V_f)$ at the
origin exists by the theorem of resolution of singularities due to
Hironaka [8]. For every
$m \in \mathbb{N}$, let $S_m$ be all points $s \in S$ such that the
equation of $V'_f$ at $s$ is of the form $z^m_1=0$ for a local coordinate
$z$ of $X$ at $s$ and denote by $\chi(S_m)$ the Euler-Poincar\'e
characteristic of $S_m$.
A'Campo proved that $s_m=m\chi(S_m)$. More precisely:
\begin{theorem}
One has

(1) $\Lambda(h^k)=\sum_{m \vert k}m \chi(S_m) $, $k \geq 1$,

(2) $\Lambda(h^0)=\chi(F_{\theta})=\sum_{m \geq 1}m \chi(S_m)$,

(3) $\mu =
\mathrm{dim}H^{n-1}(F_{\theta},\mathbb{C})=(-1)^{n-1}[-1+\sum_{m \geq
1}m \chi(S_m)]$.
\end{theorem}
Therefore the numbers $\chi(S_m)$ don't depend on the chosen resolution
and are topological invariants of the singularity.
As a consequence we have the following result that may be useful for
resolving the multiplicity conjecture.
\begin{proposition}
If $f(z)=f_k(z)+f_{k+1}(z)+ \cdots +f_{k+r}(z) $,
$g(z)=g_l(z)+g_{l+1}(z)+ \cdots +g_{l+s}(z)$ and $k+r<l$ then there is
no topological equivalence between germs of $V_f$ and $V_g$ at the
origin.
\end{proposition}
\begin{proof}
Let $h_1$ and $h_2$ be the monodromies associated to $f$ and $g$ and
$s_1,s_2, \cdots$ and $s'_1,s'_2, \cdots$ the two related sequences of $f$
and $g$ respectively as above. If there exists such an equivalence then
$\Lambda(h_1^j)=\Lambda(h_2^j)=0$ and $s_j=s'_j=0$ for every $j$. Hence
$\mu_f=\mu_g=(-1)^{n-1}[-1+\sum_{j \geq 1}s_j]=(-1)^n$. If $n$ is odd
this is impossible and if $n$ is even, then $\mu_f=\mu_g=1$. In this
case $k=l=2$. Contradiction!
\end{proof}
The second result is the following [1], [9]:
\begin{theorem}
Given two germs of hypersurfaces $V_f$ and $V_g$. Let
$\mathbb{P}C(V_f)$, respectively $\mathbb{P}C(V_g)$, denote the
projectivized tangent cone which is a subvariety of
$\mathbb{C}P^{n-1}$. If $\chi(\mathbb{C}P^{n-1} \setminus
\mathbb{P}C(V_f)) \neq 0$ and $\chi(\mathbb{C}P^{n-1} \setminus
\mathbb{P}C(V_g)) \neq 0$, then topological equisingularity of $V_f$
and $V_g$ implies $m_f=m_g$.
\end{theorem}
The key point of the proof is that: if $\chi(\mathbb{C}P^{n-1}
\setminus \mathbb{P}C(V_f)) \neq 0$ then by theorem (4.3),
$m_f=inf\{s \in \mathbb{N}| \Lambda(h^s)\neq 0 \}$.

It is unknown whether $\chi(\mathbb{C}P^{n-1} \setminus
\mathbb{P}C(V_f))$ is a topological invariant or not.

\begin{example}
Let $g=z_1^l+z_2^l + \cdots + z_n^l$, $V_g=g^{-1}(0) \subset \mathbb{C}^n$ and
$C(V_g) \subset \mathbb{C}P^{n-1}$ and $F_{\theta}$ be the fibre of the
Milnor fibration of $g$ at the origin. By (6.1) in the appendix we have
$$\mu_g=(l-1)^n.$$
With one blowing up at the origin, the singularity of $g$ may be
resolved and then we may apply the theorem (4.3): $S=\mathbb{C}P^{n-1}$
and
\begin{displaymath}
S_m=\left\{ \begin{array}{ll}
\phi & \textrm{if $m \neq l$}\\
\mathbb{C}P^{n-1} \setminus V_g & \textrm{if $m=l$.}
              \end{array}\right.
\end{displaymath}

By theorem (4.3),
$$\mu_g=(-1)^{n-1}[-1+ l\chi(S_l)].$$
The numbers $\chi(S_l)$ and $\chi (C(V_g))$ are related by
$$ \chi (S_l)+\chi (C(V_g))= \chi (\mathbb{C} P^{n-1})=n.$$ Therefore
we obtain the following well known formula
$$\chi (C(V_g))=n-\frac{1-(1-l)^n}{l} .$$
\end{example}
\begin{example}
Let $ \mathcal{A} $ be the set of all holomorphic
functions $g$ such that $0 \in \mathbb{C}^{n}$ is an isolated
singularity for the first nonzero homogeneous part of the Taylor
expansion of $g$. Then by an argument (see (6.3) in the appendix) the
origin is an isolated singularity of $g$. Let $g \in \mathcal{A}$ with
algebraic multiplicity $l$ and the leading term $g_l$. Since $g$ and
$g_l$ have the same projevtivized tangent cones and $V_{g_l}$ is
topologically equivalent to $V_{z_1^l+z_2^l + \cdots + z_n^l}$ then by
the previous example
$$ \chi(\mathbb{C}P^{n-1} \setminus
\mathbb{P}C(V_g))=\frac{1-(1-l)^n}{l} \neq 0 .$$
Hence by theorem (4.5) any topological equivalence between two elements
of $ \mathcal{A} $ preserves multiplicities.

Still it is unknown whether there is any topological equivalence between
$g \in \mathcal{A}$ and $f \notin \mathcal{A}$. By contrary if there
exists such an equivalence
then $k < l $, where $k$ and $l$ are the multiplicities of $f$ and $g$
respectively. The reason is that by (6.2) the Milnor number $\mu_f>
(k-1)^n$ and $\mu_g=(l-1)^n$ and by theorem (4.3), Milnor number is a
topological invariant. Therefore it remains to show that: Let
$g=z_1^l+z_2^l + \cdots + z_n^l$ and $f=f_k+ \cdots + f_{k+r}$ with $k
< l$ and
$ k+r \geq l $ then the germs $V_f$ and $V_g$ at the origin are not
topologically equisingular.
\end{example}

\section{On the deformation of complex hypersurfaces}
Let us, instead of dealing with a pair of hypersurfaces, consider
families of hypersurfaces, $V_{f_{t}}$, all having an isolated singular
point at the origin and depending continuously in
$t \in [0,1]$
and $f_0=f$ and $f_1=g$. We denote by $C(V_{f_{t}})$, the tangent cone
at $0$ of $V_{f_{t}}$, that is, the zero set of the initial polynomial
of $f_{t}$.
H. King generalized theorem (2.4) as follows [11]:
\begin{theorem}
Suppose $f_{t}: (\mathbb{C}^{n},0) \rightarrow (\mathbb{C},0)$, $t \in
[0,1]$ is a continuous family of holomorphic germs with the same Milnor
number and $n \neq 3$. Then there is a continuous
family of germs of homeomorphisms $h_t:(\mathbb{C}^{n},0) \rightarrow
(\mathbb{C}^{n},0) $ so that
$f_0=f_t \circ h_t$
\end{theorem}

Now we have the following result:\\
\begin{theorem}
If for every $t_0 \in [0,1]$ there exist a neighborhood $I_{t_0}$ of
$t_0$ in $[0,1]$ and a line $L_{t_0}$ through 0 in $\mathbb{C}^{n}$
such that $L_{t_0} \cap C(V_{f_{s}})=\{ 0 \}$ for
$s \in I_{t_0}$, then topological equisingularity of the family implies
equimultiplicity provided that $ n \neq 3 $.
\end{theorem}

\begin{proof}
By theorem (5.1) there exists a continuous family of homeomorphisms
$\varphi_t$ such that
$f_t=f \circ \varphi_t$. Therefore for every $t_0 \in [0,1]$ we may write
$$f_s= f_{t_0} \circ \varphi_{{t_0}s}$$
where $\varphi_{{t_0}s}= \varphi^{-1}_{t_0} \circ \varphi_s $.
Since $f_{t_0}$ is uniformly continuous on a compact small ball $B_r
\subset \mathbb{C}^{n}$ around $ 0 $, there exists
$ \eta > 0$ such that, for any $z, w \in B_r$,
$$\vert z- w \vert < \eta~~~\Longrightarrow
\vert f_{t_0}(z)- f_{t_0}(w) \vert < {\min}_{u \in S_\delta} \vert
f_{t_0}(u) \vert ,$$
where $S_\delta$ is the boundary of $\overline{D_\delta}$, the closed disc
with radius
$\delta < r/2 $ in $L_{t_0}$ around 0. Let $\varepsilon:= \mathrm{min}
\{ \eta , \delta \}$.
By continuity of $\varphi_s$, if $I_{t_0}$ is sufficiently small then
$\vert \varphi_{{t_0}s}(z) - z \vert < \varepsilon $ for $s \in
I_{t_0}$. Then for all $z$ in the closed ball $B_\delta \subset
\mathbb{C}^{n}$, $\varphi_{{t_0}s}(z) \in B_r$ and
$$\vert f_{t_0}(z)- f_{t_0} \circ {\varphi}_{t_0s}(z)  \vert <
{\min}_{u \in S_\delta} \vert f_{t_0}(u) \vert.$$
In particular for all $z \in S_\delta$ we have
$$ \vert  f_{t_0}(z)- f_s(z) \vert < \vert f_{t_0}(z) \vert, ~~~
\mathrm{for}~~~ s \in I_{t_0}.$$
By hypothesis $L_{t_0} \cap C(V_{f_{s}})=\{ 0 \}$ for $s \in I_{t_0}$,
then $m_{f_s}$ is the order at 0 of $f_s \arrowvert_{L_{t_0}} $ for $s
\in I_{t_0}$. By theorem (1.6) in [12, Ch.VI],
$f_{t_0} {\arrowvert}_{L_{t_0}}$ and $f_s \arrowvert_{L_{t_0}}$ have
the same number of zeros, counted with their multiplicities in the
interior of $\overline{D_\delta}$. As
$f_{t_0} {\arrowvert}_{L_{t_0}}$ and $f_s \arrowvert_{L_{t_0}}$ vanish
only at 0 on $\overline{D_\delta}$, the orders at 0 of $f_{t_0}
{\arrowvert}_{L_{t_0}}$ and $f_s \arrowvert_{L_{t_0}}$ are equal. So
$m_{f_{t_0}}=m_{f_s}$ for $s \in I_{t_0}$. This tells us that the multiplicity
of the deformation is constant.
\end{proof}
\section{Topological invariants for holomorphic vector fields}
Let $ U \subset \mathbb{C}^n $ be an open neighborhood of $ 0 \in
\mathbb{C}^n $ and $X: U \rightarrow \mathbb{C}^n$,
$X(0)=0$, a holomorphic vector field with a singularity at $ 0 \in
\mathbb{C}^n $. The integral curves of $X$ are complex curves i.e.
Riemann surfaces parametrized locally as the solutions of the
differential equation
$$ \dfrac{dx}{dt}=X(x),~ x \in U,~ t \in \mathbb{C}. $$
These curves define a complex one dimensional foliation $
\mathcal{F}=\mathcal{F}_X $ with singularity at
$ 0 \in \mathbb{C}^n $. We define the \textit{algebraic multiplicity}
of $X$ as the degree of its first nonzero jet, i.e. $m=m_X$ where
$$ X = X_m + X_{m+1}+ \cdots $$
is the Taylor series of $X$ and $ X_m $ is not identically zero.

In analogy with the case of hypersurfaces we define the \textit{Milnor
number} of the vector field $X$ at
$ 0 \in \mathbb{C}^n $ as
$$ \mu= \mathrm{dim}_{\mathbb{C}}\dfrac{\mathcal{O}_n}{<X_1, \cdots ,X_n>}$$
where $\mathcal{O}_n $ is the ring of germs of holomorphic functions at
$ 0 \in \mathbb{C}^n $ and \\
$ <X_1, \cdots ,X_n> $ is the ideal generated by the germs at $ 0 \in
\mathbb{C}^n $ of the coordinate functions of $X$.
This number is finite if and only if $0 \in \mathbb{C}^n$ is an
isolated singularity of $X$, a hypothesis which we will assume from now
on.

We say $\mathcal{F}_X$ is topologically equivalent with
$\mathcal{F}_{X^{\prime}}$, $X^{\prime}$ is a holomorphic vector field
defined in a neighborhood $U^{\prime}$ of $0 \in \mathbb{C}^n$, if
there is a homeomorphism
$\varphi : U \rightarrow U^{\prime} $ fixing the origin (singularity)
and sending every leaf of the foliation $\mathcal{F}_X$ into a leaf of
the foliation $\mathcal{F}_{X^{\prime}}$.

A similar question is the following: is $ m_X $ a topological invariant
of the foliation $\mathcal{F}_X$?

In a remarkable work [5], C. Camacho, A. Lins Neto and P. Sad deal with
this problem. First of all they prove the following result:
\begin{theorem}
The Milnor number of a holomorphic vector field $X$ as above is a
topological invariant provided that $n \geq 2$.
\end{theorem}
Now we restrict ourselves to $n=2$ and recall some of the features
coming from [5].
A germ of a vector field $X$ with an isolated singularity may be given by
$X = a(x,y) \frac{\partial}{\partial x} + b(x,y)
\frac{\partial}{\partial y} $ where
$a$ and $b$ are holomorphic functions with isolated zero in a
neighborhood $ U \subset \mathbb{C}^{2} $. Denote by
$ \mathcal{F} $ the foliation induced by $X$.
The main tool in the local study is the resolution theorem of Seidenberg
[21] that establishes a canonical reduction.
More precisely, there is a holomorphic map
$\pi : M \rightarrow \mathbb{C}^{2}$ obtained as a composition of a
finite number of blowing ups at points over $\{0\}$, such that at each
singular point $m \in M $ of $\widetilde{\mathcal{F}}$ the foliation with
isolated singularity constructed from $\pi^*(\omega)$ is reduced: there
are coordinate charts $(z,w)$ such that $z(m)=w(m)=0$,
$\widetilde{\mathcal{F}}$ is given locally by the expression
$A(z,w)\frac{\partial}{\partial x} + B(z,w)\frac{\partial}{\partial
y}=0$ and the Jacobian $\dfrac{\partial(A,B)}{\partial(z,w)}(0,0)$
has at least one nonzero eigenvalue. Moreover if $\lambda_1 \neq 0 \neq
\lambda_2$ are eigenvalues of the matrix then ${\lambda_1}/{\lambda_2}
\notin \mathbb{Q}_{+}$. If one of the eigenvalues of the above matrix
of a singularity is zero and another different from zero we call it
saddle-node. By definition a \textit{generalized curve} is a germ of a
vector field $X$ at $ 0 \in (\mathbb{C}^2,0)$ and singular at the
origin such that its desingularization does not admit any saddle-node. In
[5] Camacho, Lins Neto and Sad proved that this property is invariant under
topological equivalences and finally they deduced the following:
\begin{theorem}
The algebraic multiplicity of a generalized curve is a topological invariant.
\end{theorem}

Also we may say the same as theorem (3.2) for polynomial foliations.

\section{Appendix}
Let $ \mathcal{A} $ be the set of all holomorphic
functions $f$ such that $0 \in \mathbb{C}^{n}$ is an isolated singularity not
only for $f$ but also for the first nonzero homogeneous polynomial of the
Taylor expansion of the $f$. Actually if the origin is an isolated
singularity of the leading term of $f$ then the same holds for $f$.
\begin{remark} We have the following relation between multiplicity and
Milnor number of $f$ (see page 194 in [3]):
$${\mu}_f={(m_f-1)}^n.$$
\end{remark}
\begin{remark}
If $0 \in \mathbb{C}^n$ is not an isolated singularity of the first
nonzero homogeneous polynomial then $\mu_f > (m_f - 1)^n$.
\end{remark}

The following proposition is true in any dimension. But the following
proof is based on the theorem of L\'e and Ramanujam which is valid for $n
\neq 3$.

\begin{proposition}
The germ at the origin of the hypersurface defined by an
element $f \in \mathcal{A}$ with the algebraic multiplicity $k$ is
topologically equivalent with the germ at the origin of the hypersurfaces
defined by $z_1^k+ \cdots +z_n^k$.
\end{proposition}
\begin{proof}By a symbol $f \sim g$ between two germs of holomorphic
functions at the origin
we mean $V_f$ and $V_g$, two germs of hypersurfaces defined by $f$ and $g$
respectively, are topological equivalent.
Let
\begin{eqnarray}
f=f_k + f_{k + 1} + \cdots \nonumber
\end{eqnarray}
be the Taylor expansion of $f$, where $f_i$ is homogeneous polynomial
of degree $i$.
The family $(H_t)_{t \in [0,1]}\in \mathcal{A}$:
\begin{eqnarray}
H_t=f_k + tf_{k + 1} + t^2f_{k + 2} + \cdots \nonumber
\end{eqnarray} defines a $ \mu $-constant family between $f$ and $ f_k
$. So $ f \sim f_k $.

Now our task is to show $P(z) \sim (z^k_1+ \cdots +z^k_n)$ where $P(z)$
is a homogeneous polynomial of degree $k$.

\textbf{Claim:} \textit{There is a non zero complex number $\alpha $
such that 0
is an isolated singularity of $F_t(z):= (1-t)(z^k_1 + \cdots + z^k_n)+
t \alpha
P(z)$ for $t \in [0,1]$.}

\textbf{The proof of claim:} The partial derivatives of $F_t(z)$ form a system
of bihomogeneous polynomials of bidegree $(1,k-1)$:

$$\dfrac{\partial F_t}{\partial z_1}= k(1-t)z^{k-1}_1 + t \alpha
\dfrac{\partial
P}{\partial z_1}$$
$$\vdots$$
$$\dfrac{\partial F_t}{\partial z_n}= k(1-t)z^{k-1}_n + t \alpha
\dfrac{\partial
P}{\partial z_n}$$

and $V:= \mathrm{Zero}(\dfrac{\partial F_t}{\partial z_1}, \cdots
,\dfrac{\partial F_t}{\partial z_n})$ is an algebraic subset of
$\mathbb{C}P(1) \times \mathbb{C}P(n-1) $.
Now consider the projection $\pi : {\mathbb{C}P(1) \times \mathbb{C}P(n-1)}
\rightarrow \mathbb{C}P(1)$. Image of $V$, $\pi(V)$, is a
Zariski-closed subset
of
$\mathbb{C}P(1)$ (see for instance [16] Pg. 33). Since  $F_t(z)$ for
$t=0$ has the
isolated singularity, $(1:0)$ is not in the $\pi(V)$. Therefore $\pi(V)$ is
finite and there are infinitely many lines in the complement of $\pi(V)$ in
$\mathbb{C}P(1)$.
Since $P(z)$ has an isolated singularity at $0 \in \mathbb{C}^n $ we
may choose
lines passing through the origin. This means that there is $\alpha $ such that
the claim is true for every $t \in \mathbb{R}$.

\end{proof}

\section*{References}
\begin{enumerate}
\item{A'Campo, Norbert; La fonction z\^eta d'une monodromie. (French)
Comment. Math. Helv.  50  (1975), 233--248}
\item{A'Campo, Norbert; Le nombre de Lefschetz d'une monodromie.
(French)  Nederl. Akad.    Wetensch. Proc. Ser. A 76 = Indag. Math.  35
 (1973), 113--118.}
\item{ Arnold, V. I.; Gusein-Zade, S. M.; Varchenko, A. N.;
Singularities of differentiable maps. Vol. I. The classification of
critical points, caustics and wave fronts. Monographs in Mathematics,
82. Birkh?user Boston, Inc., Boston, MA, 1985.}
\item{Burau W.: Kennzeichnung der schlauchknoten. Abh. Math. Sem. Ham.
Univ., \textbf{9} (1932), 125-133. }
\item{Camacho, C\'esar; Lins Neto, Alcides; Sad, Paulo; Topological
invariants and equidesingularization for holomorphic vector fields.  J.
Differential Geom.  20  (1984),  no. 1, 143--174.}
\item{Comte, Georges; Milman, Pierre; Trotman, David; On Zariski's
multiplicity problem. Proc. Amer. Math. Soc.  130  (2002),  no. 7,
2045--2048}
\item{Eyral, Christophe; Gasparim, Elizabeth; Multiplicity of complex
hypersurface singularities, Rouch\'e satellites and Zariski's problem.
Preprint, (2005).}
\item{Hironaka, Heisuke; Resolution of singularities of an algebraic
variety over a field of characteristic zero. I, II.  Ann. of Math. (2)
79 (1964), 109--203; ibid. (2)  79  1964 205--326.}
\item{Karras, Ulrich; Equimultiplicity of deformations of constant
Milnor number. Proceedings of the conference on algebraic geometry
(Berlin, 1985),  186--209, Teubner-Texte Math., 92, Teubner, Leipzig,
1986.}
\item{King, Henry C.; Topological type of isolated critical points.
Ann. Math. (2)  107 (1978), no. 2, 385--397.}
\item{King, Henry C.; Topological type in families of germs.  Invent.
Math.  62  (1980/81), no. 1, 1--13}
\item{Lang, Serge; Complex analysis. Fourth edition. Graduate Texts in
Mathematics, 103. Springer-Verlag, New York, 1999.}
\item{L\^e D.T.; Calcul du nombre de cycles \'evanouissaints d'une
hypersurface complexe. Ann. Inst. Fourier (Grenoble) \textbf{23} (1973)
261-270.}
\item{L\^e D.T; Ramanujam C.P.; The invariance of Milnor number implies
the invariance of the topological type, Amer. J. Math. 98 (1976) 67-78.}
\item{Milnor M.; Singular points of complex hypersurfaces. Ann. Math.
Stud. \textbf{61} Princeton.}
\item{Mumford D.; Algebraic Geometry I. Complex Projective Varieties,
2nd ed., Classics in Mathematics, Springer, Berlin, 1995.}
\item{Perron, B.; Conjugaison topologique des germes de fonctions
holomorphes \`a singularit\'e isol\'ee en dimension trois. Invent.
Math. 82 (1985),  no. 1, 27--35.}
\item{Risler, Jean-Jacques; Trotman, David; Bi-Lipschitz invariance of
the multiplicity. Bull. London Math. Soc.  29  (1997),  no. 2,
200--204.}
\item{Samuel P.; Alg\'ebricit\'e  de certains points singuliers
alg\'ebroides. J. Math. Pures Appl. \textbf{35} (1965) 1-6.}
\item{Saeki, Osamu; Topological types of complex isolated hypersurface
singularities. Kodai Math. J.  12  (1989),  no. 1, 23--29.}
\item{Seidenberg, A.; Reduction of singularities of the differential
equation $ Ady=Bdx $.  Amer. J. Math.  90  1968 248--269.}
\item{Zariski O.; Some open questions in the theory of
equisingularities. Bull.
Amer. Math. Soc. \textbf{77} (1971) 481-491.}
\item{Zariski O.; Collected papars, Vol IV, MIT Press, Cambridge,
Mass., 1979.}
\end{enumerate}

\end{document}